\documentclass[11pt]{amsart}

\author[N.~Murru]{Nadir Murru}
\address{Universit\`a degli Studi di Torino, Department of Mathematics, Torino, Italy}
\email{nadir.murru@unito.it}
\urladdr{\url{http://orcid.org/0000-0003-0509-6278}}

\author[C.~Sanna]{Carlo Sanna}
\address{Universit\`a degli Studi di Torino, Department of Mathematics, Torino, Italy}
\email{carlo.sanna.dev@gmail.com}
\urladdr{\url{http://orcid.org/0000-0002-2111-7596}}


\title[On the $k$-regularity of the $k$-adic valuation of Lucas sequences]{On the $k$-regularity of the $k$-adic valuation of \\ Lucas sequences}

\usepackage{amsmath}
\usepackage{amssymb}
\usepackage{amsthm}
\usepackage{geometry}
\usepackage{hyperref}
\usepackage{enumerate}
\hypersetup{colorlinks=true}

\newtheorem{thm}{Theorem}[section]
\newtheorem{lem}[thm]{Lemma}

\def\lcm{\operatorname{lcm}}
                                 
\begin{document}

\maketitle

\begin{abstract}
For integers $k \geq 2$ and $n \neq 0$, let $\nu_k(n)$ denotes the greatest nonnegative integer $e$ such that $k^e$ divides $n$.
Moreover, let $(u_n)_{n \geq 0}$ be a nondegenerate Lucas sequence satisfying $u_0 = 0$, $u_1 = 1$, and $u_{n + 2} = a u_{n + 1} + b u_n$, for some integers $a$ and $b$. 
Shu and Yao showed that for any prime number $p$ the sequence $\nu_p(u_{n + 1})_{n \geq 0}$ is $p$-regular, while Medina and Rowland found the rank of $\nu_p(F_{n + 1})_{n \geq 0}$, where $F_n$ is the $n$-th Fibonacci number.

We prove that if $k$ and $b$ are relatively prime then $\nu_k(u_{n + 1})_{n \geq 0}$ is a $k$-regular sequence, and for $k$ a prime number we also determine its rank.
Furthermore, as an intermediate result, we give explicit formulas for $\nu_k(u_n)$, generalizing a previous theorem of Sanna concerning $p$-adic valuations of Lucas sequences.
\end{abstract}

\section{Introduction}

For integers $k \geq 2$ and $n \neq 0$, let $\nu_k(n)$ denotes the greatest nonnegative integer $e$ such that $k^e$ divides $n$.
In particular, if $k = p$ is a prime number then $\nu_p(\cdot)$ is the usual $p$-adic valuation.
We shall refer to $\nu_k(\cdot)$ as the \emph{$k$-adic valuation}, although, strictly speaking, for composite $k$ this is not a ``valuation'' in the algebraic sense of the term, since it is not true that $\nu_k(mn) = \nu_k(m) + \nu_k(n)$ for all integers $m,n \neq 0$.

Valuations of sequences with combinatorial meanings have been studied by several authors (see, e.g., \cite{AMM08, Coh99, HZZ12, Len95, Len14, ML14, PS07, San16bis, San16, SM09}).
To this end, an important role is played by the family of $k$-regular sequences, which were first introduced and studied by Allouche and Shallit~\cite{AS92, AS03b, AS03} with the aim of generalizing the concept of automatic sequences.

Given a sequence of integers $s(n)_{n \geq 0}$, its \emph{$k$-kernel} is defined as the set of subsequences
\begin{equation*}
\ker_k(s(n)_{n \geq 0}) := \{s(k^e n + i)_{n \geq 0} : e \geq 0, \; 0 \leq i < k^e\} .
\end{equation*}
Then $s(n)_{n \geq 0}$ is said to be \emph{$k$-regular} if the $\mathbb{Z}$-module $\langle \ker_k(s(n)_{n \geq 0}) \rangle$ generated by its $k$-kernel is finitely generated.
In such a case, the \emph{rank} of $s(n)_{n \geq 0}$ is the rank of this $\mathbb{Z}$-module. 

Allouche and Shallit provided many examples of regular sequences.
In particular, they showed that the sequence of $p$-adic valuations of factorials $\nu_p(n!)_{n \geq 0}$ is $p$-regular \cite[Example~9]{AS92}, and that the sequence of $3$-adic valuations of sums of central binomial coefficients 
\begin{equation*}
\nu_3\!\left(\sum_{i \,=\, 0}^n \binom{2i}{i}\right)_{n \geq 0}
\end{equation*}
is $3$-regular \cite[Example~23]{AS92}.
Furthermore, for any polynomial $f(x) \in \mathbb{Q}[x]$ with no roots in the natural numbers, Bell~\cite{Bell07} proved that the sequence $\nu_p(f(n))_{n \geq 0}$ is $p$-regular if and only if $f(x)$ factors as a product of linear polynomials in $\mathbb{Q}[x]$ times a polynomial with no root in the $p$-adic integers.

Fix two integers $a$ and $b$, and let $(u_n)_{n \geq 0}$ be the Lucas sequence of characteristic polynomial $f(x) = x^2 - ax - b$, i.e., $(u_n)_{n \geq 0}$ is the integral sequence satisfying $u_0 = 0$, $u_1 = 1$, and $u_{n + 2} = a u_{n + 1} + b u_n$, for each integer $n \geq 0$.
Assume also that $(u_n)_{n \geq 0}$ is nondegenerate, i.e., $b \neq 0$ and the ratio $\alpha / \beta$ of the two roots $\alpha, \beta \in \mathbb{C}$ of $f(x)$ is not a root of unity.

Using $p$-adic analysis, Shu and Yao~\cite[Corollary~1]{SY11} proved the following result.

\begin{thm}\label{thm:shu_yao}
For each prime number $p$, the sequence $\nu_p(u_{n+1})_{n \geq 0}$ is $p$-regular.
\end{thm}

In the special case $a = b = 1$, i.e., when $(u_n)_{n \geq 0}$ is the sequence of Fibonacci numbers $(F_n)_{n \geq 0}$, Medina and Rowland~\cite{MR15} gave an algebraic proof of Theorem~\ref{thm:shu_yao} and also determined the rank of $\nu_p(F_{n+1})_{n \geq 0}$.
Their result is the following.

\begin{thm}\label{thm:med_row}
For each prime number $p$ the sequence $\nu_p(F_{n+1})_{n \geq 0}$ is $p$-regular.
Precisely, for $p \neq 2,5$ the rank of $\nu_p(F_{n+1})_{n \geq 0}$ is $\alpha(p) + 1$, where $\alpha(p)$ is the least positive integer such that $p \mid F_{\alpha(p)}$, while for $p = 2$ the rank is $5$, and for $p = 5$ the rank is $2$.
\end{thm}

In this paper, we extend both Theorem~\ref{thm:shu_yao} and Theorem~\ref{thm:med_row} to $k$-adic valuations with $k$ relatively prime to $b$.
Let $\Delta := a^2 + 4b$ be the discriminant of $f(x)$.
Also, for each positive integer $m$ relatively prime to $b$ let $\tau(m)$ denotes the \emph{rank of apparition} of $m$ in $(u_n)_{n \geq 0}$, i.e., the least positive integer $n$ such that $m \mid u_n$ (which is well-defined, see, e.g.,~\cite{Ren13}).

Our first two results are the following.

\begin{thm}\label{thm:is_kreg}
If $k \geq 2$ is an integer relatively prime to $b$, then the sequence $\nu_k(u_{n+1})_{n \geq 0}$ is $k$-regular.
\end{thm}

\begin{thm}\label{thm:rank}
Let $p$ be a prime number not dividing $b$, and let $r$ be the rank of $\nu_p(u_{n+1})_{n \geq 0}$.
\begin{enumerate}[$\phantom{m}\bullet$]
\item If $p \mid \Delta$ then:
\begin{enumerate}[$\bullet$]
\item $r = 2$ if $p \in \{2, 3\}$ and $\nu_p(u_p) = 1$, or if $p \geq 5$;
\item $r = 3$ if $p \in \{2, 3\}$ and $\nu_p(u_p) \neq 1$.
\end{enumerate}
\item If $p \nmid \Delta$ then:
\begin{enumerate}[$\bullet$]
\item $r = 5$ if $p = 2$ and $\nu_2(u_6) \neq \nu_2(u_3) + 1$;
\item $r = \tau(p) + 1$ if $p > 2$, or if $p = 2$ and $\nu_2(u_6) = \nu_2(u_3) + 1$.
\end{enumerate}
\end{enumerate}
\end{thm}

Note that Theorem~\ref{thm:med_row} follows easily from our Theorem~\ref{thm:rank}, since in the case of Fibonacci numbers $b = 1$, $\Delta = 5$, $\nu_2(F_3) = 1$, $\nu_2(F_6) = 3$, and $\tau(p) = \alpha(p)$.

As a preliminary step in the proof of Theorem~\ref{thm:is_kreg}, we obtain some formulas for the $k$-adic valuation $\nu_k(u_n)$, which generalize a previous result of the second author.
Precisely, Sanna~\cite{San16} proved the following formulas for the $p$-adic valuation of $u_n$.

\begin{thm}\label{thm:un_padic}
If $p$ is a prime number such that $p \nmid b$, then
\begin{equation*}
\nu_p(u_n) =
\begin{cases}
\nu_p(n) + \varrho_p(n) & \emph{ if } \tau(p) \mid n , \\
0 & \emph{ if } \tau(p) \nmid n ,
\end{cases}
\end{equation*}
for each positive integer $n$, where
\begin{equation*}
\varrho_2(n) := 
\begin{cases}
\nu_2(u_3) & \emph{ if } 2 \nmid \Delta , \; 2 \nmid n , \\
\nu_2(u_6) - 1 & \emph{ if } 2 \nmid \Delta , \; 2 \mid n , \\
\nu_2(u_2) - 1 & \emph{ if } 2 \mid \Delta ,
\end{cases}
\end{equation*}
and
\begin{equation*}
\varrho_p(n) = \varrho_p := 
\begin{cases}
\nu_p(u_{\tau(p)}) & \emph{ if } p \nmid \Delta , \\
\nu_3(u_3) - 1 & \emph{ if } p \mid \Delta , \; p = 3, \\
0 & \emph{ if } p \mid \Delta , \; p \geq 5 ,
\end{cases}
\end{equation*}
for $p \geq 3$.
\end{thm}

Actually, Sanna's result~\cite[Theorem~1.5]{San16} is slightly different but it quickly turns out to be equivalent to Theorem~\ref{thm:un_padic} using \cite[Lemma~2.1(v), Lemma~3.1, and Lemma~3.2]{San16}.
Furthermore, in Sanna's paper it is assumed $\gcd(a,b) = 1$, but the proof of \cite[Theorem~1.5]{San16} works exactly in the same way also for $\gcd(a,b) \neq 1$.

From now on, let $k = p_1^{a_1} \dots p_h^{a_h}$ be the prime factorization of $k$, where $p_1 < \cdots < p_h$ are prime numbers and $a_1, \ldots, a_h$ are positive integers.

We prove the following generalization of Theorem~\ref{thm:un_padic}.

\begin{thm}\label{thm:un_kadic}
If $k \geq 2$ is an integer relatively prime to $b$, then
\begin{equation*}
\nu_k(u_n) = 
\begin{cases}
\nu_k(c_k(n) n) & \emph{ if } \tau(p_1 \cdots p_h) \mid n , \\
0 & \emph{ if } \tau(p_1 \cdots p_h) \nmid n ,
\end{cases}
\end{equation*}
for any positive integer $n$, where
\begin{equation*}
c_k(n) := \prod_{i \,=\, 1}^h p_i^{\varrho_{p_i}\!(n)} .
\end{equation*}
\end{thm}

Note that Theorem~\ref{thm:un_kadic} is indeed a generalization of Theorem~\ref{thm:un_padic}.
In fact, if $k = p$ is a prime number then obviously
\begin{equation*}
\nu_p(c_p(n) n) = \nu_p(p^{\varrho_p(n)} n) = \nu_p(n) + \varrho_p(n) ,
\end{equation*}
for each positive integer $n$.

\section{Preliminaries}

In this section we collect some preliminary facts needed to prove the results of this paper.
We begin with some lemmas on $k$-regular sequences.

\begin{lem}\label{lem:kreg_ring}
If $s(n)_{n \geq 0}$ and $t(n)_{n \geq 0}$ are two $k$-regular sequences, then $(s(n) + t(n))_{n \geq 0}$ and $s(n) t(n)_{n \geq 0}$ are $k$-regular too.
Precisely, if $A$ is a finite set of generators of $\langle \ker_k(s(n)_{n \geq 0}) \rangle$ and $B$ is a finite set of generators of $\langle \ker_k(t(n)_{n \geq 0}) \rangle$, then $A \cup B$ is a set of generators of $\langle \ker_k((s(n) + t(n))_{n \geq 0}) \rangle$.
\end{lem}
\begin{proof}
See~\cite[Theorem 2.5]{AS92}.
\end{proof}

\begin{lem}\label{lem:kreg_linsub}
If $s(n)_{n \geq 0}$ is a $k$-regular sequence, then for any integers $c \geq 1$ and $d \geq 0$ the subsequence $s(cn + d)_{n \geq 0}$ is $k$-regular.
\end{lem}
\begin{proof}
See~\cite[Theorem~2.6]{AS92}.
\end{proof}

\begin{lem}\label{lem:kreg_periodic}
Any periodic sequence is $k$-regular.
\end{lem}
\begin{proof}
An ultimately periodic sequence is $k$-automatic for all $k \geq 2$, see~\cite[Theorem~5.4.2]{AS03b}. A $k$-automatic sequence is $k$-regular, see~\cite[Theorem~1.2]{AS92}.
\end{proof}

\begin{lem}\label{lem:kreg_sr}
Let $s(n)_{n \geq 0}$ be a sequence of integers.
If there exist some
\begin{equation}\label{equ:srinkker}
s_1 = s, s_2, \ldots, s_r \in \langle \ker_k(s(n)_{n \geq 0}) \rangle
\end{equation}
such that the sequences $s_j(k n + i)_{n \geq 0}$, with $0 \leq i < k$ and $1 \leq j \leq r$, are $\mathbb{Z}$-linear combinations of $s_1, \ldots, s_r$, then $s(n)_{n \geq 0}$ is $k$-regular and $\langle \ker_k(s(n)_{n \geq 0}) \rangle$ is generated by $s_1, \ldots, s_r$.
\end{lem}
\begin{proof}
It is sufficient to prove that $s(k^e n + i)_{n \geq 0} \in \langle s_1, \ldots, s_r \rangle$ for all integers $e \geq 0$ and $0 \leq i < k^e$.
In fact, this claim implies that $\langle \ker_k (s(n)_{n \geq 0}) \rangle \subseteq \langle s_1, \ldots, s_r \rangle$, while by (\ref{equ:srinkker}) we have $\langle s_1, \ldots, s_r \rangle \subseteq \langle \ker_k (s(n)_{n \geq 0}) \rangle$, hence $\langle \ker_k (s(n)_{n \geq 0}) \rangle = \langle s_1, \ldots, s_r \rangle$ and so $s(n)_{n \geq 0}$ is $k$-regular.
We proceed by induction on $e$. For $e = 0$ the claim is obvious since $s = s_1$.
Suppose $e \geq 1$ and that the claim holds for $e - 1$.
We have $i = k^{e-1} j + i^\prime$, for some integers $0 \leq j < k$ and $0 \leq i^\prime < k^{e-1}$.
Therefore, by the induction hypothesis,
\begin{align*}
s(k^e n + i)_{n \geq 0} &= s(k^{e-1} (kn + j) + i^\prime)_{n \geq 0} \\
 &\in \langle s_1(kn + j)_{n \geq 0}, \ldots, s_r(kn + j)_{n \geq 0} \rangle \\
 &\subseteq \langle s_1, \ldots, s_r \rangle ,
\end{align*}
and the claim follows.
\end{proof}

The next lemma is well-known, we give the proof just for completeness.

\begin{lem}\label{lem:kreg_lin}
The sequence $\nu_k(n + 1)_{n \geq 0}$ is $k$-regular of rank $2$.
Indeed, $\langle \ker_k(\nu_k(n + 1)_{n \geq 0}) \rangle$ is generated by $\nu_k(n+1)_{n \geq 0}$ and the constant sequence $(1)_{n \geq 0}$.
\end{lem}
\begin{proof}
For all nonnegative integers $n$ and $i < k$ we have
\begin{equation*}
\nu_k(kn + i + 1) = \begin{cases}
1 + \nu_k(n+1) & \text{ if } i = k - 1 , \\
0 & \text{ if } i < k - 1 .
\end{cases}
\end{equation*}
Therefore, putting $s_1 = \nu_k(n + 1)_{n \geq 0}$ and $s_2 = (1 + \nu_k(n + 1))_{n \geq 0}$ in Lemma~\ref{lem:kreg_sr}, we obtain that $\langle \ker_k(\nu_k(n + 1)_{n \geq 0}) \rangle$ is generated by $\nu_k(n+1)_{n \geq 0}$ and $(1 + \nu_k(n + 1))_{n \geq 0}$, hence it is also generated by $\nu_k(n+1)_{n \geq 0}$ and $(1)_{n \geq 0}$, which are obviously linearly independent.
Thus $\nu_k(n + 1)_{n \geq 0}$ is $k$-regular of rank $2$.
\end{proof}

Now we state a lemma that relates the $k$-adic valuation of an integer with its $p_i$-adic valuations.
The proof is quite straightforward and we leave it to the reader.

\begin{lem}\label{lem:kadic_min}
We have
\begin{equation*}
\nu_k(m) = \min_{i=1, \ldots, h} \left\lfloor \frac{\nu_{p_i}(m)}{a_i}\right\rfloor ,
\end{equation*}
for any integer $m \neq 0$.
\end{lem}

We conclude this section with two lemmas on the rank of apparition $\tau(n)$.

\begin{lem}\label{lem:tau_div}
For each prime number $p$ not dividing $b$, 
\begin{equation*}
\tau(p) \mid p - (-1)^{p-1}\left(\frac{\Delta}{p}\right) ,
\end{equation*}
where $\left(\tfrac{\cdot}{p}\right)$ denotes the Legendre symbol.
In particular, if $p \mid \Delta$ then $\tau(p) = p$.
\end{lem}
\begin{proof}
The case $p = 2$ is easy. 
For $p > 2$ see \cite[Lemma~1]{Som80}.
\end{proof}

\begin{lem}\label{lem:tau_lcm}
If $m$ and $n$ are two positive integers relatively prime to $b$, then 
\begin{equation*}
\tau(\lcm(m,n)) = \lcm(\tau(m), \tau(n)) .
\end{equation*}
\end{lem}
\begin{proof}
See \cite[Theorem~1(a)]{Ren13}.
\end{proof}

\section{Proof of Theorem~\ref{thm:un_kadic}}

Thanks to Lemma~\ref{lem:kadic_min}, we know that
\begin{equation}\label{equ:un_min}
\nu_k(u_n) = \min_{i=1, \ldots, h} \left\lfloor \frac{\nu_{p_i}(u_n)}{a_i}\right\rfloor .
\end{equation}
Moreover, from Lemma~\ref{lem:tau_lcm} it follows that
\begin{equation*}
\tau(p_1 \cdots p_h) = \lcm\{\tau(p_1), \ldots, \tau(p_h)\} .
\end{equation*}
Therefore, on the one hand, if $\tau(p_1 \cdots p_h) \nmid n$ then $\tau(p_i) \nmid n$ for some $i \in \{1, \ldots, h\}$, so that by Theorem~\ref{thm:un_padic} we have $\nu_{p_i}(u_n) = 0$, which together with (\ref{equ:un_min}) implies $\nu_k(u_n) = 0$, as claimed.

On the other hand, if $\tau(p_1 \cdots p_h) \mid n$ then $\tau(p_i) \mid n$ for $i=1,\ldots,h$.
Hence, from (\ref{equ:un_min}), Theorem~\ref{thm:un_padic}, and Lemma~\ref{lem:kadic_min}, we obtain
\begin{equation*}
\nu_k(u_n) = \min_{i=1, \ldots, h} \left\lfloor \frac{\nu_{p_i}(n) + \varrho_{p_i}(n)}{a_i}\right\rfloor = \min_{i=1, \ldots, h} \left\lfloor \frac{\nu_{p_i}(c_k(n) n)}{a_i}\right\rfloor = \nu_k(c_k(n) n) ,
\end{equation*}
so that the proof is complete.

\section{Proof of Theorem~\ref{thm:is_kreg}}

Clearly, if $k$ is fixed, then $c_k(n)$ depends only of the parity of $n$.
Thus it follows easily from Theorem~\ref{thm:un_kadic} that
\begin{equation}\label{equ:un_kadic_sum}
\nu_k(u_{n+1}) = \nu_k(c_k(1) (n+1)) \, s(n) + \nu_k(c_k(2) (n+1)) \, t(n) ,
\end{equation}
for each integer $n \geq 0$, where the sequences $s(n)_{n \geq 0}$ and $t(n)_{n \geq 0}$ are defined by
\begin{equation*}
s(n) := 
\begin{cases}
1 & \text{ if } \tau(p_1 \cdots p_2) \mid n + 1, \; 2 \nmid n + 1, \\
0 & \text{ otherwise},
\end{cases}
\end{equation*}
and
\begin{equation*}
t(n) := 
\begin{cases}
1 & \text{ if } \tau(p_1 \cdots p_2) \mid n + 1, \; 2 \mid n + 1, \\
0 & \text{ otherwise} .
\end{cases}
\end{equation*}
On the one hand, by Lemma~\ref{lem:kreg_lin} and Lemma~\ref{lem:kreg_linsub}, we know that both $\nu_k(c_k(1) (n+1))_{n \geq 0}$ and $\nu_k(c_k(2) (n+1))_{n \geq 0}$ are $k$-regular sequences.
On the other hand, by Lemma~\ref{lem:kreg_periodic}, also the sequences $s(n)_{n \geq 0}$ and $t(n)_{n \geq 0}$ are $k$-regular, since obviously they are periodic. 

In conclusion, thanks to (\ref{equ:un_kadic_sum}) and Lemma~\ref{lem:kreg_ring}, we obtain that $\nu_k(u_{n+1})_{n \geq 0}$ is a $k$-regular sequence.

\section{Proof of Theorem~\ref{thm:rank}}

First, suppose that $p \mid \Delta$.
By Lemma~\ref{lem:tau_div} we have $\tau(p) = p$.
Moreover, it is clear that $\varrho_p(n) = \varrho_p$ does not depend on $n$.
As a consequence, from Theorem~\ref{thm:un_padic} it follows easily that
\begin{equation}\label{equ:sum1}
\nu_p(u_{n+1}) = \nu_p(n + 1) + s(n) ,
\end{equation}
for any integer $n \geq 0$, where the sequence $s(n)_{n \geq 0}$ is defined by
\begin{equation*}
s(n) := \begin{cases}
\varrho_p & \text{ if } n + 1 \equiv 0 \bmod p , \\
0 & \text{ if } n + 1 \not\equiv 0 \bmod p .
\end{cases}
\end{equation*}
On the one hand, if $p \in \{2, 3\}$ and $\nu_p(u_p) = 1$, or if $p \geq 5$, then $\varrho_p = 0$.
Thus $s(n)_{n \geq 0}$ is identically zero and it follows by (\ref{equ:sum1}) and Lemma~\ref{lem:kreg_lin} that $r = 2$.
On the other hand, if $p \in \{2, 3\}$ and $\nu_p(u_p) \neq 1$, then $\varrho_p \neq 0$.
Moreover, for $i=0,\ldots,p-1$ we have
\begin{equation*}
s(pn + i) = \begin{cases} 
\varrho_p & \text{ if } i = p - 1 , \\
0 & \text{ if } i \neq p - 1 ,
\end{cases}
\end{equation*}
hence from Lemma~\ref{lem:kreg_sr} it follows that $s(n)_{n \geq 0}$ is $p$-regular and that $\langle\ker_p(s(n)_{n \geq 0})\rangle$ is generated by $s(n)_{n \geq 0}$ and $(\varrho_p)_{n \geq 0}$.
Therefore, by (\ref{equ:sum1}), Lemma~\ref{lem:kreg_lin}, and Lemma~\ref{lem:kreg_ring}, we obtain that $\nu_p(u_{n+1})_{n \geq 0}$ is a $p$-regular sequence and that $\langle\ker_p(\nu_p(u_{n+1})_{n \geq 0})\rangle$ is generated by $\nu_p(n+1)_{n \geq 0}$, $s(n)_{n \geq 0}$, and $(1)_{n \geq 0}$, which are clearly linearly independent, hence $r = 3$. 

Now suppose $p \nmid \Delta$.
By Lemma~\ref{lem:tau_div}, we know that $p \equiv \varepsilon \bmod \tau(p)$, for some $\varepsilon \in \{-1, +1\}$.
Furthermore, if $p = 2$ then it follows easily that $\tau(2) = 3$.
As a consequence, from Theorem~\ref{thm:un_padic} we obtain that
\begin{equation}\label{equ:sum2}
\nu_p(u_{n+1}) = s(n) + t(n) ,
\end{equation}
for any integer $n \geq 0$, where the sequences $s(n)_{n \geq 0}$ and $t(n)_{n \geq 0}$ are defined by
\begin{equation*}
s(n) := \begin{cases}
\nu_p(n+1) + v & \text{ if } n + 1 \equiv 0 \bmod \tau(p) \\
0 & \text{ if } n + 1 \not\equiv 0 \bmod \tau(p) ,
\end{cases}
\end{equation*}
with $v := \nu_p(u_{\tau(p)})$, and
\begin{equation*}
t(n) := \begin{cases}
\nu_2(u_6) - \nu_2(u_3) - 1 & \text{ if } p = 2, \; n + 1 \equiv 0 \bmod 6 , \\
0 & \text{ otherwise}. 
\end{cases}
\end{equation*}

We shall show that $s(n)_{n \geq 0}$ is a $p$-regular sequence of rank $\tau(p) + 1$.
Let us define the sequences $s_j(n)_{n \geq 0}$, for $j = 0,\ldots,\tau(p)-1$, by
\begin{equation*}
s_j(n) := \begin{cases}
1 & \text{ if } n + j + 1 \equiv 0 \bmod \tau(p) , \\
0 & \text{ if } n + j + 1 \not\equiv 0 \bmod \tau(p) .
\end{cases}
\end{equation*}
On the one hand, for $i = 0, \ldots, p - 2$ we have
\begin{align*}
s(pn+i) &= \begin{cases}
\nu_p(pn + i + 1) + v & \text{ if } pn + i + 1 \equiv 0 \bmod \tau(p) , \\
0 & \text{ if } pn + i + 1 \not\equiv 0 \bmod \tau(p) ,
\end{cases} \\
&= \begin{cases}
v & \text{ if } \varepsilon n + i + 1 \equiv 0 \bmod \tau(p) , \\
0 & \text{ if } \varepsilon n + i + 1 \not\equiv 0 \bmod \tau(p) ,
\end{cases} \\
&= \begin{cases}
v & \text{ if } n + (\varepsilon(i + 1) - 1) + 1 \equiv 0 \bmod \tau(p) , \\
0 & \text{ if } n + (\varepsilon(i + 1) - 1) + 1 \not\equiv 0 \bmod \tau(p) ,
\end{cases} \\
&= v \cdot s_{(\varepsilon(i + 1) - 1) \bmod \tau(p)}(n) ,
\end{align*}
since $p \nmid i + 1$ and consequently $\nu_p(pn + i +1) = 0$.

On the other hand,
\begin{align}\label{equ:ss0}
s(pn + p - 1) &= \begin{cases}
\nu_p(pn+p) + v & \text{ if } p(n + 1) \equiv 0 \bmod \tau(p) , \\
0 & \text{ if } p(n + 1) \not\equiv 0 \bmod \tau(p) ,
\end{cases} \\
&= \begin{cases}
\nu_p(n+1) + v + 1 & \text{ if } n + 1 \equiv 0 \bmod \tau(p) , \\
0 & \text{ if } n + 1 \not\equiv 0 \bmod \tau(p) ,
\end{cases} \nonumber\\
&= s(n) + s_0(n) , \nonumber
\end{align}
since $\nu_p(pn + p) = \nu_p(n + 1) + 1$ and $\gcd(p, \tau(p)) = 1$.

Furthermore, for $i = 0,\ldots, p - 1$ and $j = 0,\ldots, \tau(p) - 1$,
\begin{align*}
s_j(pn+i) &= \begin{cases}
1 & \text{ if } pn + i + j + 1 \equiv 0 \bmod \tau(p) , \\
0 & \text{ if } pn + i + j + 1 \not\equiv 0 \bmod \tau(p) , 
\end{cases} \\
&= \begin{cases}
1 & \text{ if } n + (\varepsilon(i + j + 1) - 1) + 1 \equiv 0 \bmod \tau(p) , \\
0 & \text{ if } n + (\varepsilon(i + j + 1) - 1) + 1 \not\equiv 0 \bmod \tau(p) , 
\end{cases} \\
&= s_{(\varepsilon(i + j + 1) - 1) \bmod \tau(p)}(n) .
\end{align*}
Summarizing, the sequences $s(pn+i)_{n \geq 0}$ and $s_j(pn+i)_{n \geq 0}$, for $i = 0,\ldots,p-1$ and $j=0,\ldots,\tau(p) - 1$, are $\mathbb{Z}$-linear combinations of $s(n)_{n \geq 0}$ and $s_j(n)_{n \geq 0}$.

Moreover, for $i = 0, \ldots, p^2 - 1$ we have
\begin{align}\label{equ:s0p2}
s_0(p^2n+i) &= \begin{cases}
1 & \text{ if } p^2n + i + 1 \equiv 0 \bmod \tau(p) , \\
0 & \text{ if } p^2n + i + 1 \not\equiv 0 \bmod \tau(p) , 
\end{cases} \\
&= \begin{cases}
1 & \text{ if } n + i + 1 \equiv 0 \bmod \tau(p) , \\
0 & \text{ if } n + i + 1 \not\equiv 0 \bmod \tau(p) , 
\end{cases} \nonumber\\
&= s_{i \bmod \tau(p)}(n) , \nonumber
\end{align}
hence, by (\ref{equ:s0p2}) and (\ref{equ:ss0}), it follows that
\begin{align}\label{equ:sitaup}
s_{i \bmod \tau(p)}(n)_{n \geq 0} &= s_0(p^2n+i)_{n \geq 0} \\
&= s(p^3 n + pi + p - 1)_{n \geq 0} - s(p^2 n + i)_{n \geq 0} \nonumber\\
&\in \langle \ker_p(s(n)_{n \geq 0}) \rangle . \nonumber
\end{align}
Since $\tau(p) \mid p - \varepsilon$, we have 
\begin{equation*}
\tau(p) \leq p - \varepsilon \leq p + 1 < p^2 , 
\end{equation*}
hence by (\ref{equ:sitaup}) we get that $s_j(n)_{n \geq 0} \in \langle \ker_p(s(n)_{n \geq 0}) \rangle$, for each $j=0,\ldots,\tau(p)-1$.

Therefore, in light of Lemma~\ref{lem:kreg_sr}, we obtain that $s(n)_{n \geq 0}$ is a $p$-regular sequence and that $\langle\ker_p(s(n)_{n \geq 0})\rangle$ is generated by $s(n)_{n \geq 0}$ and $s_j(n)_{n \geq 0}$, with $j = 0,\ldots,\tau(p)-1$.
It is straightforward to see that these last sequences are linearly independent, hence $s(n)_{n \geq 0}$ has rank $\tau(p) + 1$.

If $p > 2$, or if $p = 2$ and $\nu_2(u_6) = \nu_2(u_3) + 1$, then $t(n)_{n \geq 0}$ is identically zero, thus from (\ref{equ:sum2}) and the previous result on $s(n)$ we find that $r = \tau(p) + 1$.

So it remains only to consider the case $p = 2$ and $\nu_2(u_6) \neq \nu_2(u_3) + 1$.
Recall that in such a case $\tau(2) = 3$, and put $d := \nu_2(u_6) - \nu_2(u_3) - 1$.
Obviously, the sequence $t(2n)_{n \geq 0}$ is identically zero, while
\begin{align*}
t(2n + 1) &= \begin{cases}
d & \text{ if } \; 2n + 2 \equiv 0 \bmod 6 , \\
0 & \text{ if } \; 2n + 2 \not \equiv 0 \bmod 6 ,
\end{cases} \\
&= \begin{cases}
d & \text{ if } \; n + 1 \equiv 0 \bmod 3 , \\
0 & \text{ if } \; n + 1 \not \equiv 0 \bmod 3 ,
\end{cases} \\
&= d \cdot s_0(n) .
\end{align*}
Thus, again from Lemma~\ref{lem:kreg_sr}, we have that $t(n)$ is a $2$-regular sequence and that $\langle \ker_p(t(n)_{n \geq 0}) \rangle$ is generated by $t(n)_{n \geq 0}$ and $d \cdot s_j(n)_{n \geq 0}$, for $j=0,1,2$.

In conclusion, by (\ref{equ:sum2}) and Lemma~\ref{lem:kreg_ring}, we obtain that $\nu_p(u_{n+1})_{n \geq 0}$ is a $2$-regular sequence and that $\langle \ker_p(\nu_p(u_{n+1})_{n \geq 0}) \rangle$ is generated by $s(n)$, $t(n)$, and $s_j(n)$, for $j=0,1,2$, which are linearly independent, hence $r = 5$.
The proof is complete.

\section{Concluding remarks}

It might be interesting to understand if, actually, $\nu_k(u_{n + 1})_{n \geq 0}$ is $k$-regular for every integer $k \geq 2$, so that Theorem~\ref{thm:is_kreg} holds even by dropping the assumption that $k$ and $b$ are relatively prime.
A trivial observation is that if $k$ and $b$ have a common prime factor $p$ such that $p \nmid a$, then $p \nmid u_n$ for all integers $n \geq 1$, and consequently $\nu_k(u_{n + 1})_{n \geq 0}$ is $k$-regular simple because it is identically zero. 
Thus the nontrivial case occurs when each of the prime factors of $\gcd(b,k)$ divides $a$.

Another natural question is if it is possible to generalize Theorem~\ref{thm:rank} in order to say something about the rank of $\nu_k(u_{n + 1})_{n \geq 0}$ when $k$ is composite.
Probably, the easier cases are those when $k$ is squarefree, or when $k$ is a power of a prime number.

We leave these as open questions to the reader.

\providecommand{\bysame}{\leavevmode\hbox to3em{\hrulefill}\thinspace}

\end{document}